\documentclass{amsart}
\usepackage{amssymb,graphicx,rotating,multirow,multicol}
\usepackage{amsmath,amsthm,amsfonts,amscd}
\usepackage[toc,page]{appendix}
\usepackage{caption,comment}
\usepackage[mathscr]{euscript}
\usepackage{epstopdf}
\usepackage{float}

\usepackage{graphicx}

\usepackage[utf8]{inputenc}
\usepackage[T2A,T1]{fontenc}
\usepackage[russian,english]{babel}
\usepackage{epigraph}
\swapnumbers
\newtheorem{theorem}{Theorem}[subsection]

\newtheorem{lemma}[theorem]{Lemma}

\newtheorem{cor}[theorem]{Corollary}
\newtheorem{proposition}[theorem]{Proposition}
\theoremstyle{remark}
\newtheorem{remark}[theorem]{Remark}
\numberwithin{equation}{subsection}

\def\NN{\mathscr{N}}

{\catcode`\@=11 \gdef\mnote#1{\marginpar{\footnotesize
 \tolerance\@M\spaceskip2.6\p@ plus10\p@ minus.9\p@\rm#1}}}
\def\Dg:{\endgraf{\bf Dg:\enspace}\ignorespaces}

\let\Bbb\mathbb
\let\Cal\mathcal
\DeclareMathOperator{\ind}{I^e}
\def\D{\Delta}
\def\G{\Gamma}
\def\g{\gamma}
\def\P{\Cal P}

\def\sm{\smallsetminus}
\DeclareMathOperator{\sgn}{sign}

\newcommand{\be}{\begin{equation}}
\newcommand{\ee}{\end{equation}}
\let\ge\geqslant 
\let\le\leqslant 

\let\la\langle
\let\ra\rangle
\let\til\widetilde

\def\Z{\Bbb Z}
\def\R{\Bbb R}
\def\C{\Bbb C}

\def\X{\Cal X}
\def\L{\Lambda}
\def\DX{\D^{\X}}
\def\XR{\X_\R}
\def\DXR{\DX_{\R}}

\def\Sym{\operatorname{Sym}}
\def\Hom{\operatorname{Hom}}
\def\SO{\operatorname{SO}}
\def\M{\Cal M^*}
\def\MM{\Cal M}
\def\DM{\D^{\MM}}
\def\Sw{\operatorname{S^{loc}}}
\def\P{\Cal P}
\def\DP{\D^{\P}}
\def\DPP{\D^{\P,1}}
\def\DPPinf{\D^{\infty,1}}
\def\PR{\Cal P_\R}
\def\DPR{\D^{\P}_\R}
\def\DPinf{\D^{\infty}}
\def\pr{\operatorname{pr}}
\def\Sec{\operatorname{Sec}}
\def\PGL{\operatorname{PGL}}
\def\Jet{\operatorname{Jet}}

\def\crem{\operatorname{Cr}}

\def\pol{\operatorname{pol}}

\makeatletter
\newcommand{\addresseshere}{%
  \enddoc@text\let\enddoc@text\relax
}
\makeatother

\begin{document}

\renewcommand{\arraystretch}{1.2}
\title[Segre indices and Welschinger wieghts]
{Segre indices and Welschinger weights as options
for  invariant count of real lines}
\author[]{ S.~Finashin, V.~Kharlamov}
\address{
 Department of Mathematics, Middle East Tech. University\endgraf
 06800 Ankara Turkey}
 \email{serge@metu.edu.tr}
\address{Universit\'{e} de Strasbourg et IRMA (CNRS)\endgraf 7 rue Ren\'{e}-Descartes, 67084 Strasbourg Cedex, France}
\email{kharlam@math.unistra.fr}


\date{}

\begin{abstract}
In our previous paper  \cite{abundance} we have elaborated a certain signed count of real lines on real hypersurfaces of degree $2n-1$ in $P^{n+1}$.
Contrary to the honest "cardinal" count, it is independent of the choice of a hypersurface, and by this reason provides
a strong lower bound on the honest count. In this count the contribution of a line is its local input
to the Euler number of a certain auxiliary vector bundle. The aim of this paper is to present other, in a sense more geometric, interpretations of
this local input.
One of them results from
a generalization of Segre species of real lines on cubic surfaces and another from
a generalization of Welschinger weights of real lines on quintic threefolds.
\end{abstract}

\maketitle
{\begin{otherlanguage*}{russian}
\setlength\epigraphwidth{.40\textwidth}
\epigraph{
Ужасно интересно \\
Все то, что неизвестно;\\
Ужасно неизвестно, \\
Все то, что интересно.
}
{Г.~Остер, \\
\textit{из мультфильма \\ "Тридцать Восемь Попугаев'' }
\footnotetext{Humorous version of Tacit's "Omne ignotum pro magnifico est"; composed by G.Oster, \textit{for the cartoon "Thirty Eight Parrots"} }
}\end{otherlanguage*}
}

\rightline
{\vbox{\hsize80mm
\noindent\baselineskip10pt{\it\footnotesize
\vskip3mm
\noindent
\rm
}}}
\vskip5mm

\section{Introduction}\label{intro}

\subsection{The subject} Let $X$ be a generic real hypersurface of degree $2n-1, n\ge 2$, in
the real projective space of dimension $n+1$. Denote by $\NN_\C$ and $\NN_\R$ the number of complex
and, respectively, real lines on $X$. These numbers are finite, since $X$ was chosen generic.
The first number
$\NN_\C$ depends only on $n$, while $\NN_\R$ depends on the choice of $X$. For example, if $n=2$ (the case of cubic surfaces)
and $X$ is non-singular, then $\NN_\R$ may take values $3, 7, 15,$ and $27$.

It was shown in \cite{abundance} and \cite{OT} that $\NN_\R\ge (2n-1)!!$ for any $n$.
The proof
was based on the following signed count of the real lines that
makes the total sum independent of $X$.
A polynomial
defining $X$ yields a section
of the  symmetric power $\operatorname{Sym}^{2n-1}(\tau^*_{2,n+2})$ of
the tautological covariant vector bundle $\tau^*_{2,n+2}$ over the real Grassmannian $G_\R(2, n+2)$ and zeros
of this section are precisely the real lines $l\subset X$.
The local Euler numbers $\ind(l,X)=\pm1$ of the zeros sum up to the total Euler number  $\NN^e_\R$ of $\operatorname{Sym}^{2n-1}(\tau^*_{2,n+2})$
(which is independent of $X$)
and we obtain immediately
$\NN_\R\ge \vert \NN^e_\R\vert$, whereas
$\NN^e_\R=(2n-1)!!$.

This brought up a natural question: {\it What
can be a direct geometric interpretation of these local indices $\ind(l,X)$} ?

In \cite{abundance}  only a partial answer to it
was given.
It was shown that
for cubic surfaces $X$ (the case $n=2$)
$\ind(l,X)$ coincides with the {\it Welschinger weight} $W(l,X)$ of $l$ on $X$,
as well as with its Segre index $S(l,X)$ expressing numerically Segre's division of lines in two species, elliptic and hyperbolic.
In a nutshell, for cubic surfaces, $W(l,X)$ is equal to
$e^{\frac{\pi i}{2} (q(l)-1)}=\pm1$ where $q: H_1(X_\R,\Z/2)\to \Z/4$
is the quadratic function representing the "canonical" $Pin^-$ structure induced on $X_\R$ from $P^3_\R$,
while the Segre index is
$e^{\frac{\pi i}2 s}=\pm1$ where $s$ is the number of fixed points of the
involution $l_\R\to l_\R$
traced out on $l_\R$ by the conics which are residual intersections
of $X$ with the hyperplanes containing $l$.

But the case $n>2$ was left open. As for the Welschinger weight, its definition for lines on higher dimensional
varieties did not cause apparent difficulties,
but our proof of $\ind(l,X)=W(l,X)$ resisted an immediate generalization because some particular properties of dimension $2$ were used.
Concerning
the Segre index,
it was not even clear how to extend the definition from $n=2$ to $n>2$.

In this paper we solve the both problems. Namely, we produce a simplified, not appealing to any auxiliary $Pin$-structure, version of Welschinger weights
of real lines and prove that they coincide with $\ind(l,X)$ (Theorem \ref{Euler=Welsh}).
In what concerns the second problem, we introduce the notion of {\it Segre index}  that generalizes Segre's division of lines into
elliptic and hyperbolic and has a transparent geometric meaning, and then prove that it
also coincides with $\ind(l,X)$ (Theorem \ref{Euler=Segre}).

Our definition of the Segre index, like Segre's definition for lines on cubic surfaces, is purely algebraic. But starting from $n\ge 3$ it is no more "one-move" definition.
We start from replacing a real line $l\subset X$
 by a real rational curve of degree $2n-2$ in $P^{n-1}$ that
describes the first jet of $X$ along $l$. We
look at all real $(n-3)$-dimensional $(2n-4)$-secants of this curve,
associate an involution on $P^1$ with each of these secants, and define
a weight (equal to $\pm 1$) of a secant depending on the reality of the fixed points (see details in Section \ref{def-S}).
Finally, we define
the Segre index to be the product of the latter weights.

The proofs of Theorems \ref{Euler=Welsh} and \ref{Euler=Segre} are independent of each other,
but based on the same strategy: we prove that for all the three types of indices (local Euler number, Welschinger weight, and
Segre index) considered as functions on
the space of curves mentioned above satisfy the same wall crossing rules. After this, it remains to check their coincidence on an example.

We hope that these results and approaches may shade a new light on the hidden structures behind this sort
of enumerative invariants, including, for example, the built-up in \cite{3spaces}
invariant signed count of odd dimensional real projective planes on projective hypersurfaces.
It is also interesting to understand how
the Segre index
is related to the quadratic form of
C.~Okonek and A.~Teleman (\cite{OT}, formula (16)) extracted from the Jacobian matrix
$A_C$ (see Subsection \ref{discriminants} below).

For related arithmetic and probabilistic aspects of counting lines on hypersurfaces, we send
the reader to \cite{KW} and, respectively, \cite{BLetc}.

\subsection{Structure of the paper}
In Section 2 we define the Welschinger weight $W(l,X)$ and in Section 3 we prove that it coincides with the local Euler number $\ind(l,X)$.
In Section 4 we switch to consideration of  multisecants to rational curves and
present a few technical results related to the so-called  {\it Castelnuovo count} of multisecants of codimension 2.
They are used in Section \ref{def-S} to define the Segre index
$S(l,X)$
and to prove its equality with $\ind(l,X)$.
Finally, in Section 6, we present several other interpretations
of the Segre index:
first
for lines on quintic threefodls and on septic fourfolds, and then in general.

\subsection{Conventions}
In this paper algebraic varieties by default are complex; for example, $P^n$ stands for the complex projective space.
By $\Sym^{k}(P^1)$, $k\ge1$,
we denote the $k$-th symmetric power of $P^1$, or equivalently, the set of effective divisors
of degree $k$ on $P^1$. When $C$ is a parametrized curve $C: P^1\to P^n$
and $M\subset P^n$ is a divisor, the notation $M\cdot C$ is used as an abbreviation for the intersection product considered
as a divisor on $P^1$.

A projective variety $X\subset P^n$ is
called {\it real} if it is invariant under the complex conjugation in $P^n$.
For a real variety $X$, we denote by $X_\R$
the set of complex points of $X$ that are fixed by the complex conjugation.
Speaking on non-singular real varieties, we mean that the whole $X$ (not only $X_\R$) has no singular points.

A complex (holomorphic) vector bundle $\pi:W\to X$ is called {\it real}, if $X$, $W$, and $\pi$ are defined over the reals,
and the real structure in $W$ is anti-linear.
Similar conventions are applied to all other algebraic notions, like splitting of a vector bundle, deformations, isomorphisms etc.

Spaces $P^n$ and bundles $\Cal O_{P^n}(n)$ are always equipped with the canonical real structures.

\subsection{Acknowledgements}
A strong impulse to this study
came from a short, but illuminating conversation of the first author with Ilya Zakharevich.
We thank also Fedor Zak
for helpful advices on manipulating secant spaces, and Alex Degtyarev for providing us a reference to a real version of Birkhoff-Grothendieck theorem.
A significant part of this work
was carried out during our joint visits to the Max Planck Institute for Mathematics
as well as during visits of the first author to the Strasbourg University, while a final touch was given during our joint visit
to the Istanbul Center for Mathematical Sciences,
and we wish to thank these institutions for hospitality
and excellent working conditions.

The second author was partially funded by the grant ANR-18-CE40-0009 of {\it Agence Nationale de Recherche}.

\section{Du c\^ot\'e de chez Welschinger}\label{def-W}

In this section we assume that $X\subset P^{n+1}$ is a real
hypersurface of degree $2n-1$ and $l\subset X$ is a real
line which does not contain any singular point of $X$.

\subsection{Balancing condition}
Over $\C$, due to Birkhoff-Grothendieck theorem in its standard version,
the normal bundle $ N_l$ of $l$ in $X$ splits into a sum of line bundles,
$ N_l=\oplus_{i=1}^{n-1} L_i,\, L_i=\Cal O_l({m_i}),$ where
$$ m_1+\dots+m_{n-1}= n+2-(2n-1)-2=-(n-1)
$$
by the adjunction formula.
Under the usual, descending order, convention, the list
of integers $m_1\ge \dots\ge m_{n-1}$ depends only on $N_l$ and is called the {\it splitting type} of $N_l$.
The splitting itself is uniquely defined up to multiplication by non-degenerate upper block-triangular matrices $A$ whose elements
$a_{ij}$ are 0 if $m_i-m_j<0$ and homogeneous polynomials of degree $m_i-m_j$ in two variables if $m_i-m_j\ge 0$.
The vector bundle $N_l$ and the line $l$ are called
{\it balanced} or {\it stable}, if $\vert m_i- m_j\vert \le 1$ for each pair $(i,j)$.  In our case, $N_l$ is balanced if and only if $m_i=-1$
for each $i$.

This traditional terminology is motivated by the fact that the codimension of a given splitting type in the versal deformation space of vector bundles over $P^1$
is equal to the sum of $m_i- m_j - 1$ taken over $m_i-m_j\ge 2$
(see, {\it e.g.}, \cite{briesk} or \cite{donin}). It can be derived then that
the splitting type of a vector bundle is preserved
under deformations if
the vector bundle is balanced and, conversely,
a splitting of a balanced vector bundle
extends to any local deformation of this bundle.

\subsection{Welschinger weights}\label{Wweights}
Let us assume that
$X$ and $l$ are both real. Then, the bundle $N_l$ is also real and,
according to the real version of Birkhoff-Grothendieck theorem (see \cite{splitting} for a statement and a proof over any field),
a splitting $ N_l=\oplus_{i=1}^{n-1} L_i$ seen as an isomorphism of complex vector bundles $N_l \to \oplus^{n-1}_{i=1}\Cal O_l({m_i})$
can be chosen real with respect to the standard real structure in $\oplus^{n-1}_{i=1}\Cal O_l({m_i})$.
Similarly to the complex version, this isomorphism is unique up to
multiplication by non-degenerate upper block-triangular matrices $A$  whose elements $a_{ij}$
are 0 if $m_i-m_j<0$ and real homogeneous polynomials of degree $m_i-m_j$ in two variables if $m_i-m_j\ge 0$.
Furthermore, if a real vector bundle is balanced, its real splitting locally extends
to any real deformation of the vector bundle.

Choose
an auxiliary $(n-1)$-subspace $H\subset P^{n+1}_\R$ disjoint from $l_\R$ (say, the subspace dual to $l$ with respect to the Fubini-Study metric)
and identify it with $P^{n-1}_\R$. Then, the splitting $N_l=\oplus_{i=1}^{n-1}L_i$ yields on the real locus $l_\R\subset X_\R$ a framing formed by
real projective lines $\nu_i(t)\subset P^{n+1}_\R$, $i=1,\dots, n-1$,  $t\in l_\R$. Namely, the
line $\nu_i(t)$ joins 
the point $t\in l_\R$ with the point 
where $H$ meets the projective 2-plane that contains $l$ and whose tangent 
plane at $t$
projects to $L_i$ in $N_l$.
Thus, we obtain an $(n-1)$-tuple $s(t)=([\nu_1(t)],\dots,[\nu_{n-1}(t)])$
of points $[\nu_i(t)]=\nu_i(t)\cap H$, which is
{\it projectively non-degenerate}, that is, spanning  $(n-2)$-subspace in $H=P^{n-1}_\R$ for each $t\in l_\R$.
As $t$ varies, $s(t)$ form a loop
in the space of such projectively non-degenerate $(n-1)$-tuples in $P^{n-1}_\R$.

\lemma\label{null-homotopic}
If all $m_i$ are odd {\rm(}in particular, if line $l$ is balanced{\rm),} the
loop $s(t)$, $t\in l_\R$, lifts to a loop of $(n-1)$-tuples of linear independent
vectors in the sphere $S^{n-1}$ that covers $P^{n-1}_\R$.
\endlemma

\proof
A loop in $P^{n+1}_\R$ represented by $l_\R$ is lifted by the covering
$\phi:S^{n+1}\to P^{n+1}_\R$ to
a half-circle, $S^1_+\subset S^{n+1}$.
Consider also the lifting to $S^{n+1}$ of a vector field tangent to $\nu_i$, $i\in\{1,\dots,n-1\}$, namely, a field
 $e_i(\theta)\in T_\theta(S^{n+1}), \theta\in S_+^1$, such that
$d\phi (e_i(\theta))$ is tangent to $\nu_i(t)\subset P_\R^{n+1}, t=\phi(\theta)$.
At each point $\theta\in S_+^1$ the
vectors tangent to $S^{n+1}$ and normal to $S_+^1$ are
parallel to the hyperplane $\R^n$ generated by $S^{n-1}$, and thus
we can identify vectors $e_i(\theta)$ with corresponding vectors in above $\R^n$.
For each $1\le i\le n-1$, since $m_i$ is odd, the real line vector bundle generated by $\nu_i(t)$ over $l_\R$ is non-orientable.
Thereby,
the path $(e_1(\theta), \dots, e_{n-1}(\theta))$ (with $e_i$ considered as vectors in above $\R^n$) is a loop of $(n-1)$-tuples of linear independent vectors in $S^{n-1}$.
By construction, this loop covers the loop $s$.
\endproof

Thus, under the assumptions of Lemma \ref{null-homotopic},
the loop $s$ is lifted to a loop of $(n-1)$-frames of linear independent vectors
in $S^{n-1}\subset\R^n$, which can be made orthogonal by
Gramm-Scmidt orthogonalization and after completing to $n$-frame,
yields a loop $\til s(t)$ in $\SO_n$, whose homotopy class we denote
$[\til s]\in\pi_1(\SO_n)$
(this group is $\Z/2$ for $n\ge 3$, and $\Z$ for $n=2$).
We define the {\it Welschinger weight} of $l$ as
$$W(l,X)=(-1)^{[\til s]}\in \{+1, -1\}.$$

Due to Proposition \ref{welldefined} below, this weight
is independent of all the choices made during the construction of the loop $\til s$.

\proposition\label{welldefined}
If all $m_i$ are odd {\rm(}in particular, if $l$ is balanced{\rm),}
the Welschinger weight $W(l,X)$ is well defined.
\endproposition

\proof
The loop $s$, whose choice depends on a Birkhoff-Grothendieck splitting, is defined up to
pointwise multiplication of $s$ by a  homotopy trivial
loop in $GL(n-1, \R)$ (due to block-triangular nature of the automorphism group of $N_l$).
The subsequent lifting of the points $[\nu_i(t)]$ to $S^{n-1}$ is defined up to conjugations, whereas
orthogonalisation of the framing is a canonical operation.
Hence, the class $[\til s]\in\pi_1(\SO_n)$ is independent of all the choices made.
\endproof

\begin{remark}
The assumption on $m_i$ made in Proposition \ref{welldefined} and Lemma \ref{null-homotopic} is trivially satisfied for $n=2$.
\end{remark}

\section{Proof of $W(l,X)=\ind(l,X)$}\label{W=I}

In this section we fix a real line $l\subset P_\R^{n+1}$, $n\ge 2$, and
also a real
coordinate system $u,v,x_1,\dots,x_n$ in $P_\R^{n+1}$ such that $l=\{x_1=\dots=x_n=0\}$.

\subsection{Background}\label{background}
Each homogeneous
polynomial $F=F(u,v,x_1,\dots, x_n)$, $\deg F=2n-1$,
defining in $P^{n+1}$ a hypersurface  containing $l$ can be presented as
\begin{equation}\label{tag}
F=x_1p_1(u,v)+\dots+x_np_n(u,v) +Q
(u,v,x_1,\dots,x_n),
\end{equation}
where  homogeneous degree $2n-2$ polynomials $p_k, 1\le k\le n$, are uniquely defined
and $Q$ vanishes to order $2$ along $l$.
We denote by $\Cal H$  the projective space of all such hypersurfaces and
equip it with the standard projective coordinates, the coefficients of $F$. Inside $\Cal H$ we
consider a subset, $\X$, formed by hypersurfaces that are non-singular at each point of $l$.
Both $\Cal H$ and $\X$ bear natural real structures, and their real points
represent real hypersurfaces.

\begin{lemma}\label{connected}
For any $n\ge 2$, $\X$ is a Zarisky open subset of $\Cal H$ (in particular, $\X$ is a smooth irreducible quasi-projective variety), and, for any $n\ge 3$, $\X_\R$
is a connected smooth manifold.
\end{lemma}

\begin{proof} A hypersurface defined by a polynomial $F$ as above is non-singular at each point of $l$ if and only if the polynomials
$p_k(u,v), 1\le k\le n$, have no common zeros. Hence, the complement of $\X$ in $\Cal H$ is a Zarisky closed subset of codimension $n-1$.
\end{proof}

In the projective space
$P(\C_{2n-2}[u,v]\otimes \C^{n})$
of all $n$-tuples $[p_1:\dots : p_n]$
of degree $2n-2$ homogeneous polynomials $p_i=p_i(u,v)$, we consider
a Zariski open subset $\P$ formed by $n$-tuples of polynomials having no common roots.
Under this condition such polynomials define
a parametrized rational curve $C\!: P^1 \to P^{n-1}$, $[u:v]\mapsto [p_1:\dots p_n]$, of degree $2n-2$,
and we can view $\P$ as the space of such curves.

We consider the natural projection $\Cal H\to
P(\C_{2n-2}[u,v]\otimes \C^{n})$, $F\mapsto[p_1:\dots:p_n]$,
and denote its restriction by $\Jet^1_l\!:\X\to\P$. All these spaces and maps are defined over the reals.
\begin{proposition}\label{discr-comparison}
Both $\Jet^1_l : \X\to\P$ and $(\Jet^1_l)_\R : \X_\R\to\P_\R$ are fibrations
 between smooth varieties with contractible fibers.
\end{proposition}

\begin{proof}
Polynomials $Q$ involved in (\ref{tag}) form a vector space, which yields the contractibility.
Smoothness of $\X, \X_\R$ and $\P, \P_\R$ follows from their openness in smooth varieties.
\end{proof}

\subsection{The discriminants}\label{discriminants}
Next, we present the components $p_k, 1\le k\le n$, of $C\in\P$ in the form
$$p_k(u,v)=a_{0,k}u^{2n-2}+\dots+a_{2n-2,k}v^{2n-2}
$$
and consider the $2n\times2n$ matrix
$$A_C=\begin{pmatrix}
a_{0,1}&0&a_{0,2}&0&\dots&a_{0,n}&0\\
a_{1,1}&a_{0,1}&a_{1,2}&a_{0,2}&\dots&a_{1,n}&a_{0,n}\\
\dots\\
a_{2n-2,1}&a_{2n-3,1}&a_{2n-2,2}&a_{2n-3,2}&\dots&a_{2n-2,n}&a_{2n-3,n}\\
0&a_{2n-2,1}&0&a_{2n-2,2}&\dots&0&a_{2n-2,n}
\end{pmatrix}.$$
Denote by $\DP$ the subvariety
of $\P$ defined by equation $\det A_C=0$
and let $\DX=(\Jet^1_l)^{-1}(\DP)$ be the corresponding subvariety of $\X$.

\begin{proposition}\label{E=sign} Set theoretically,
$\Delta^{\Cal X}$
is formed by those $X\in \X$ for which the normal bundle $N_{l,X}$
of $l$ in $X$ is not balanced.
For each $X\in\XR\sm\DXR$,
the Euler index $\ind(l,X)$ is equal to $\sgn(\det A_C)=\pm1$.
\end{proposition}

Recall that $\ind(l,X)$
is the local Euler index
at $l\in G_\R(2,n+2)$
of the section $s_F$ of $\Sym^{2n-1}(\tau_{2,n+2}^*)$ determined by $F$.
The Proposition \ref{E=sign} then follows immediately from the following Lemma.

\begin{lemma}\label{fundamental-matrix-conditions}
For $X\in \X$ and its defining polynomial $F$, the following hold:
\begin{enumerate}
\item The matrix $A_C$ is the Jacobi matrix of $s_F$ at the point $l\in Gr(2, n+2)$.
\item The normal bundle $N_{l,X}$ of $l$ in $X$ is balanced if and only if $\det A_C\ne0$.
\item The determinant $\det A_C$ vanishes if and only if there exists a non-zero $n$-tuple of
linear polynomials $L_i=L_i(u,v)$, $i=1,\dots,n$, such that the dot-product
 $p\cdot L=p_1L_1+\dots+p_nL_n$ vanishes as a polynomial.
\end{enumerate}
\end{lemma}

\begin{proof}{(cf. \cite{galois})} Straightforward calculation of the Jacobi matrix in standard local coordinates on $Gr(2,n+2)$ at the point $[l]$
gives the first statement.
The second statement follows from identification of $N_{l,X}$ as a sheaf
with the kernel of the map $O(1)^{\oplus n}\to O(2n-1)$ given by matrix product with the vector $(p_1,\dots, p_n)$.
Combining of similar terms expresses vanishing of
$p_1L_1+\dots+p_nL_n$ as
non-triviality of the kernel of the map $w\mapsto A_Cw$, $w\in\C^{2n}$,
which yields part (3).
\end{proof}

Lemma \ref{fundamental-matrix-conditions} implies also the following result.

\begin{proposition}\label{D-reduced}
If $n\ge 3$, then the discriminants $\DP\subset\P$ and
$\DX\subset\X$ are non-empty, reduced and irreducible
hypersurfaces.
For $n=2$, they are empty.
\end{proposition}

\proof
If $n=2$, then $\det A_C$ is the resultant of $p_1, p_2$, while the case of pairs $p_1,p_2$ having a common zero is excluded
by definition of $\P$ and $\Cal X$.
From now on, we assume that $n\ge 3$.

To prove that $\DP\subset P$ and
$\DX\subset\X$ have no multiple components we show that
$\det A_C$ considered as a polynomial in variables $a_{i,j}$ is reduced.
Since every variable $a_{k,i}$ enters in each of the monomials of $\det A_C$ in degree $\le 2$,
we can present $\det A_C$ as a product $F^2G$, where the polynomial $G$ is reduced.
Furthermore, note the following.

(1) Since $\det A_C$ is symmetric
with respect to permutations, $\sigma$, of polynomials $p_i$,
each of $F$ and $G$
is either alternating or symmetric with respect to the
induced simultaneous permutations of variables, $a_{k,i}\mapsto a_{k,\sigma(i)}$, $0\le k\le 2n-2$, $1\le i\le n$.
In particular, if $F$ or $G$ contains a monomial with $a_{k,i}$ (in some power)
then it contains monomials with $a_{k,j}$ (in the same power) for all values $1\le j\le n$.

(2) A variable $a_{k,i}$
enters into $F$ if and only if it does not enter in $G$. This is because it enters into $\det A_C$
at most quadratically.

(3) If $a_{k,i}$ enters into $F$, then its both ``neighbors'' $a_{k\pm1,i}$
does not enter in $G$.
This is because $a_{k,i}^2$ will appear in $F^2$, and if, say, $a_{k+1,i}$ enters in $G$,
then $\det A_C$ would contain a term with $a_{k,i}^2a_{k+1,i}$ which is impossible, because all three factors come from the same pair of rows.

(1) - (3) imply that either $F$ or $G$ is constant. If $F$ is constant, then $\det A_C$ is reduced.
The second option is impossible, because $\det A_C$ restricted to $\X_\R$
is not of constant sign as it follows from
examples like those
with $p_k, 1\le i\le n-1,$ of the form
$u^{2n-2}+au^{2n-1}v, u^{2n-1}v + u^{2n-2}v^2, v^4, \dots, v^{2n-2}$.

The same examples show that $\det A_C$ is not identically zero,
and we conclude that $\DP\subset\P$ is a non empty reduced hypersurface as well as
$\DX\subset \X$.

To prove irreducibility of $\DP$,
we consider $C=[p_1:\dots : p_n]\in\DP$ and use Lemma \ref{fundamental-matrix-conditions}(3)
to find a non-zero $n$-tuple $L_1,\dots,L_n$ of linear polynomials such that
$p_1L_1+\dots+p_nL_n=0$.
Using that the linear system spanned by $L_i$ has rank $\le2$, we transform the latter identity
by an appropriate change of coordinates $x_1,\dots,x_n$ into
$p_1'L_1'+p_2'L_2'=0$, where $p_1'$ and $p_2'$ are the corresponding linear combinations of $p_1,\dots,p_n$. Therefore, we consider
in the affine space of $(n+1)$-tuples of homogeneous polynomials $l_1(u,v), l_2(u,v), r(u,v), q_3(u,v), \dots, q_n(u,v)$ of degrees $1,1,2n-3,
2n-2,\dots, 2n-2$ respectively a Zariski open subset $W$
formed by $(n+1)$-tuples whose component-polynomials have no common roots.
The map
$$
\phi : PGL(n)\times W\to \DP,\quad (M, (l_1,l_2,r, q_3, \dots, q_n))\mapsto (l_1r,l_2r,q_3,\dots, q_n)M
$$
has an irreducible domain and
is a morphism dominant
up to a codimension $\ge 2$ subvariety of
$\DP$ formed by $[p_1:\dots :p_n]$ with linear dependent components.
This implies
irreducibility $\DP$, since $\DP$ is a hypersurface in a nonsingular variety, $\X$, and hence pure dimensional.

Irreducibility of $\DX$ follows from that of $\DP$ by Lemma \ref{discr-comparison}.
\endproof

\subsection{Wall-crossing}
\begin{lemma}\label{continuity-W}
Functions $\ind$ and $W$ are continuous {\rm (}locally constant{\rm )}
on $\XR\sm\DXR$.
\end{lemma}

\begin{proof} For $\ind$, it follows from  Proposition \ref{E=sign}.
For $W$, it follows from stability of real balanced vector bundles under small real deformations and preserving the non-singularity of $X$
along $l$ under variations of $X$ in $\X$.
\end{proof}

For $n\ge 3$,
the hypersurface $\DX$ has a natural stratification in terms of splitting types of $N_l$. In this paper, we restrict our attention to the main, open,
stratum $\DX_0$ that corresponds to the splitting type $O_l\oplus O_l(-2)\oplus O_l(-1)^{\oplus (n-3)}$.

As is known (and straightforward to check; see, for example, \cite{galois}), this stratum is a non-empty open Zariski subset of $\DX$ and all the other strata (formed by the more deep
splitting types, that is the types different from $O_l\oplus O_l(-2)\oplus O_l(-1)^{\oplus (n-3)}$ and $O_l(-1)^{\oplus (n-1)}$) form a closed codimension 2
subvariety of $\DX$.

In the real setting, we mean
by  {\it walls} in $\XR$ the top-dimensional connected components of $(\DX_{0})_\R$, and by {\it chambers} the connected components of $\XR\sm\DXR$.

\begin{proposition}\label{wall-crossing-of-I-and-W}
For any $n\ge 3$, each of the functions $\ind$ and $W$
 takes opposite values on the opposite sides of each wall of the space $\XR$
{\rm (}that is, they alternate their values as long as a path $X_t\in\XR$ crosses transversally a wall{\rm ).}
\end{proposition}

\begin{proof}
For $\ind$, it follows from Lemma \ref{fundamental-matrix-conditions} and Proposition \ref{D-reduced}. For $W$, we argue as follows ({\it cf.} proof of Proposition 3.5 in \cite{spinor}).

Recall that
the vector bundle $O_l\oplus O_l(-2)\oplus O_l(-1)^{\oplus (n-3)}$
admits a universal deformation  $\Cal E'\to D$ with base $D=\{ t\in\C : \vert t\vert<1\}$. It
splits in a direct sum of the trivial family over $D$ with fiber $O_l(-1)^{\oplus (n-3)}$
and a universal deformation $\Cal E\to D$ of $O_l\oplus O_l(-2)$.
The deformation $\Cal E$
is obtained from  two trivial vector bundles,
$$
\begin{aligned}
&\C^{2}\times (P^1\sm \{[1:0]\}) \times D \to (P^1\sm\{ [1:0]\})\times D \text{ and }\\
& \C^{2}\times( P^1\sm\{ [0:1]\})\times D \to (P^1\sm\{ [0:1]\})\times D,
\end{aligned} $$
 by gluing $(v, [z:1],t)=(Gv,[ 1:z^{-1}],t)$ with the transition matrix
$$
G=
\begin{pmatrix}
z^{2}&tz\\
0 &1\end{pmatrix}.
$$
For each $t\ne 0$ the bundle $\Cal E_t$ splits into a sum $\Cal O_l(-1)\oplus\Cal O_l(-1)$.
In the chart $\C=\{[z:1]\}\subset P^1$ such a splitting is defined by the sections
$z\mapsto(t,-z)$ and
$z\mapsto (0,1)$ (in the second chart, by $w\mapsto (0,-1)$ and $w\mapsto (t,w)$, respectively).


Note that the universal bundles $\Cal E$ and $\Cal E'$ carry natural real structures.
So, for any $X_0\in (\DX_0)_\R$ and some real neighborhood $U\subset  \X_\R$,
$X_0\in U$,
there exists a real map $\phi:U\to D$ such that the bundles $N_{l,X}$ with $X\in U$ are induced by $\phi$ from $\Cal E'$.
As is shown in \cite{galois}, there exists a point in $\DX_0$
for which such a mapping $\phi$ is a submersion at $X_0$.
By continuity argument, and using the irreducibility of
$\DX$ (see Proposition \ref{D-reduced}),
 we conclude that
 $\phi$ is a submersion for any choice of $X_0$ (complex or real). Hence, there exists a real slice $\sigma : D\to U$ such that
for $X_t =\sigma(t)$, $t\in D$, $t\ne0$, the
bundle $N_{l, X_t}$
is isomorphic to
the direct sum
$\Cal E_t\oplus O_l(-1)^{\oplus (n-3)}$.

The meromorphic section of $\Cal E$ defined in the first chart of $P^1\times D$ by $([z:1],t)\mapsto (t,-z)\in \C^2$
is transversal to the zero section. Switching to real loci and smooth category, we trivialize the family of real vector bundles
$(\Cal E_t)_\R$, $t\in D_\R$. From the above transversality we
deduce
that under this trivialization
the sections defined along the real axis of the first chart of $P^1$ by
$x\mapsto(t,-x)$ and $x\mapsto(-t,-x)$  with a fixed $t\ne 0$ differ by a full twist. Hence, $[{\tilde s}(t) ]=1+ [{\tilde s}(-t)] \mod 2$. Herefrom the alternation of $W$.\end{proof}

\begin{proposition}\label{one-example}
For any $n\ge2$, there exists
$X\in\XR\sm\DXR$ with $\ind(l,X)=W(l,X)$.
\end{proposition}

\begin{proof}
Let  $X$ be defined by equation $x_1u^{2n-2}+x_2u^{2n-4}v^2+\dots + x_nv^{2n-2}$.
Then, $\det A_C=1$. Therefore, the normal bundle $N_{l}$ of $l$ in $X$ is balanced and $\ind(l,X)=1$ (see Lemma \ref{fundamental-matrix-conditions}).
An explicit splitting of $N_l$ is given by a direct sum of the normal bundles of $l$ in the following ruled surfaces
$Y_i\subset X$, $l\subset Y_i$, $i=1,\dots,n-1$:
$$
x_1=\dots = x_{i-1}=0, \quad x_iu^2+x_{i+1}v^2=0,\quad x_{i+2}=\dots=x_n=0
$$
(note that each of $Y_i$ is nonsingular along $l$).
In the notation of \ref{Wweights}, the points $[\nu_i(t)]\in P^{n-1}$ given by this splitting at $t=[u:v:0:\dots:0]\in l$
 are
$$
[v^2:-u^2:0:\dots:0], [0:v^2:-u^2:0:\dots:0],\dots, [0:\dots:v^2:-u^2]
$$
and this framing for $t\in l_\R$ is
homotopic to the constant one, namely,
$$[1:-1:0:\dots: 0], [0:1:-1:\dots:0], \dots, [0:\dots:1:-1].
$$
Hence, $[\til s]=0$ which gives $W(l,X)=1$ and we are done.
\end{proof}

\begin{theorem}\label{Euler=Welsh} The equality $\ind(l,X)=W(l,X)$ holds for each $X\in\XR\sm\DXR$.
\end{theorem}

\begin{proof} For $n=2$, see \cite{abundance}. For $n\ge 3$, it is immediate from
Propositions \ref{wall-crossing-of-I-and-W}, \ref{one-example}
and Lemmas \ref{connected}, \ref{continuity-W}.
\end{proof}

\section{Spaces of multisecants}

\subsection{Multisecants}\label{secants}
For any $C\in\P$, an $r$-dimensional subspace $M\subset P^{n-1}$ is called an {\it $r$-dimensional $k$-secant of $C$}
if the improper intersection divisor  $D=M\cdot C$  (for its definition see, for instance, [Vogel]) has degree $k$,
that is, $D\in
\Sym^k (P^1)$. We denote by $\Sec^r_k(C)$ the set of all such secants, and let $\Sec^r_{\ge k}(C)=\cup_{m\ge k}\Sec^r_m(C)$.

In this paper, we are interested in
$\Sec_{2n-4}^{n-3}(C)$ which is finite for a generic $C\in \P$ (see Proposition \ref{Castelnuovo} below),
and sometimes in
$\Sec^r_k(C)$ with neighboring values $r=n-3, n-4$ and $k=2n-3, 2n-2$.
For instance, we have
the following interpretation of the discriminant $\DP$ introduced in Subsection
\ref{discriminants}.

\begin{lemma}\label{syzygie}
$\DP=\{C\in\P\,|\,\Sec_{\ge2n-3}^{n-3}(C)\ne\varnothing \}.$
\end{lemma}

\begin{proof}
We need to show that for any $C\in\P$ its matrix $A_C$ (see Subsection \ref{discriminants}) has
$\det A_C=0$ if and only if $\Sec_{k}^{n-3}(C)\ne\varnothing$ for some $k\ge 2n-3$.
By Lemma \ref{fundamental-matrix-conditions}(3), if $\det A_C=0$, we have
$p_1L_1+\dots+p_nL_n =0$ for some linear polynomials $L_i=L_i(u,v)$.
Then, by a linear change of coordinates $x_1,\dots,x_n$, we make vanish all the polynomials $L_i$ except two, say $L_1$ and $L_2$,
and get vanishing of $p_1L_1+p_2L_2$.
Such vanishing implies that $p_1$ and $p_2$ have $2n-3$ common roots. These roots provide $2n-3$ points on
$C$, which are contained in the $(n-3)$-subspace $x_1=x_2=0$.
This argument works obviously in the opposite direction too.
\end{proof}

In what follows we need also to deal with the following auxiliary spaces:
 $$
 \begin{aligned}
 \DPinf=\{C\in\P\,|\,\dim & \Sec_{2n-4}^{n-3}(C)\ge1\},\, \DPPinf=\{C\in\P\,|\,\Sec_{2n-4}^{n-4}(C)\ne\varnothing\},\\
 &  \DPP=\{C\in\P\,|\,\Sec_{2n-2}^{n-3}(C)\ne\varnothing\}.
 \end{aligned}
$$
Note that one can define equivalently
$\DPP$  as the space of curves $C\in\P$ whose image is contained inside a hyperplane of $P^{n-1}$,
or, in other words, as having linearly dependent polynomial components $p_i$, $i=1,\dots, n$.
It is also trivial to see that both $\DPP$ and $\DPPinf$ lie in $\DP\cap\DPinf$.

\subsection{Multisecants via projection}\label{via_projection}
The complete linear system $\Cal O(2n-2)$ on $P^1$ embeds $P^1$ as a rational normal curve
in the projective space
$P(V^*)$
where $V^*$ is dual to $V=H^0(P^1,\Cal O(2n-2))$. We identify $P(V^*)$ with $P^{2n-2}$
and denote
this embedding $\G:P^1\to P^{2n-2}$.
Note that, for any rational curve $C\in\P\sm\DPP$, $C:P^1\to P^{n-1}$, of degree $2n-2$,
the $(n-2)$-plane $\L_C\subset P^{2n-2}$ dual to the linear system defining $C:P^1\to P^{n-1}$
is disjoint from $\Gamma$.

 In the inverse direction, if we fix a projective subspace $H\subset P^{2n-2}$ of dimension
 $n-1$
 which is disjoint from $\G$ and identify $H$ with $P^{n-1}$, we
obtain a subvariety $\P_H$ formed by those $C(\L)\in\P\sm\DPP$ which are obtained by projection of $\G$ from
$(n-2)$-subspaces $\L\subset P^{2n-2}$
with $\L\cap(\G\cup H)=\varnothing$:
 \begin{equation}\label{E:phi-def}
 C: {P}^1\ \stackrel{\Gamma}{\longrightarrow}\
  {P}^{2n-2}\ \stackrel{\pi_\L}{\relbar\rightarrow}\ H = {P}^{n-1}
 \end{equation}
where $\pi_\L$ is the projection to $H$ centered at $\L$. We can summarize it as follows.

\begin{proposition}\label{C-L-correspondence}
The correspondence $C\mapsto \L_C$
defines a projection $\P\sm\DPP\to G(n-1,2n-1)$ whose image
is an open subset $U_\G\subset G(n-1,2n-1)$ represented
by $(n-2)$-subspaces $\L\subset P^{2n-2}$ which do not intersect $\G$,
and the fibers are formed by projectively equivalent curves $C$.

For a fixed subspace $H\subset P^{2n-2}$ and a fixed isomorphism $H=P^{n-1}$,
the map $\L\mapsto C(\L)$ gives a section of the above fibration over an open subset $U_H\subset U_\G$
formed by $\L$ disjoint from $H$.
\qed\end{proposition}

Since $\Gamma$ is a rational normal curve,  its multisecants have a particular property:
for any $D\in \Sym^k(P^1)$ there exists one and only one $(k-1)$-plane $W_D$ with $W_D\cdot \G=D$.

\begin{lemma}\label{MtoW}
For any $D\in \Sym^k(P^1)$, $k\le 2n-2$, the image $M_D=\pi_\L(W_D)$ is
 a multisecant of $C=\pi_\L(\G)$ of dimension $(k-1)-\dim(W_D\cap \L)-1$.
In particular, for $k=2n-4$, we have: \begin{enumerate}
\item $\dim M_D\ge n-4$.
\item
$M_D\in\Sec^{n-3}_{\ge 2n-4}(C)$ if and only if $W_D\cap\L$ has codimension $1$ in $\L$.
\item
$M_D\in\Sec^{n-4}_{2n-4}(C)$ if and only if $\L\subset W_D$.
\end{enumerate}
For $k=2n-3$, we have:
\begin{enumerate}\item[(4)]
$\dim M_D\ge n-3$.
\item[(5)] $M_D\in \Sec^{n-3}_{2n-3}(C)$ if and only if $\L\subset W_D$.
\end{enumerate}
\end{lemma}

\begin{proof}
Straightforward from standard linear algebra dimension formulae.
\end{proof}

\begin{lemma}\label{W-transversality}
Consider two divisors $D\in \Sym^{2n-3}(P^1)$ and $D'\in\Sym^{2n-4}(P^1)$, $n\ge3$, such that at least two points of $D'$ are different from
the points of $D$. Then the subspaces $W_D$ and $W_{D'}$ intersect transversely in $P^{2n-2}$.
\end{lemma}

\begin{proof}
Note that $\dim W_D=2n-4$ and $\dim W_{D'}=2n-5$ and the inclusion
$W_{D'}\subset W_D$ would imply
$D'=W_{D'}\cap \G\subset D =W_{D}\cap \G$,
which contradicts to our assumption.

Therefore, it is left to rule out
the possibility of $\dim(W_D\cap W_{D'})=\dim W_D-2$.
If this is the case, then $W_D$ and $W_{D'}$ span together a hyperplane $H\supset W_D\cup W_{D'}$, so that $H\cap\G$
includes $2n-3$ points of $D$ and at least two additional points from $D'$,
which contradicts to Bezout theorem, since $\G$ is of degree $2n-2$ and lying in no hyperplane.
\end{proof}

\subsection{Around the Castelnuovo count
}\label{around-count}

\begin{lemma}\label{generic-on-DP}
If $C$ is a generic point in $\DP$, then:
\begin{enumerate}
\item
The set $\Sec_{\ge 2n-2}^{n-3}(C)$ is empty and so
$\Sec_{\ge 2n-3}^{n-3}(C)=\Sec_{2n-3}^{n-3}(C)$.
\item
The set $\Sec_{2n-3}^{n-3}(C)$ contains only one element.
\item
For the secant $M\in \Sec_{2n-3}^{n-3}(C)$,
the points of divisor $D=C\cdot M$ are all distinct
and are in general linear position in $M$.
\item
The set $\Sec_{\ge 2n-5}^{n-4}(C)$ is empty for $n\ge4$.
\end{enumerate}
\end{lemma}

\begin{proof} (1) Note that $\Sec_{\ge 2n-2}^{n-3}(C)$ can not be nonempty for a generic $C\in\DP$, since,
otherwise,
by Bezout theorem,
this would imply that $C$ is contained in a hyperplane
and, hence, is not generic in $\DP$, as any element of $\DPP$.

(2) Assume that $\Sec_{2n-3}^{n-3}(C)$ contains
two secants, $M_1$ and $M_2$.
In the case of $\dim (M_1\cap M_2)= n-5$, we can choose coordinates $x_1, \dots, x_n$ in $P^{n-1}$ so that
$$M_1=\{x_1=x_2=0\},\quad M_2=\{x_3=x_4=0\}.$$
In terms of polynomial components $p_1,\dots,p_n$ of $C$
the condition that
$M_1, M_2$ are $(2n-3)$-secants means that
$p_1$ has $2n-3$ common roots with $p_2$, while $p_3$ has $2n-3$ common roots with $p_4$.
This gives
the dimension
$$2(2n-3+2\times2)+(n-4)(2n-1)+\dim G(4,n)+2\dim G(2,4)-1=2n^2-n-3
$$
for the space of such curves $C$, which is less than $\dim\DP=\dim\P-1=n(2n-1)-2$.
Therefore, such $C$ is not generic in $\DP$.

If $\dim (M_1\cap M_2)= n-4$, then we choose coordinates so that $M_1$ is defined by $x_1=x_2=0$ and
$M_2$ by $x_2=x_3=0$.
This time the dimension count gives us
$$2n-3+3\times 2 + (n-3)(2n-1)+2\dim G(2,3)+ \dim G(3,n) -1=2n^2-2n,$$
which is also strictly less than $\dim \DP$ for $n\ge3$.

(3)
 A similar dimension count in the case of one secant $M\in\Sec_{2n-3}^{n-3}$ gives
$$\dim \DP =2n-3 + 2\times 2 + (n-2)(2n-1) + \dim G(2,n)-1=2n^2-n-2,$$
where the summand $2n-3$ drops if
the divisor $D$ has multiple points and the summand $(n-2)(2n-1)$ drops if
the points of $D$ are not in a general linear position in $M$.

(4) If  $M$ is an element of $\Sec_{\ge 2n-5}^{n-4}(C)$, then in an appropriate coordinate system
the polynomials $p_1, p_2, p_3$ have at least $2n-5$ common roots, and once more the result follows from
a dimension count, which gives us
$2n-5+3\times 4 + (n-3)(2n-1) +\dim G(3, n)-1= 2n^2-2n<\dim \DP$.
\end{proof}

Let us consider some auxiliary cycles in $G(2n-4, 2n-1)$. One of them, $A_\Gamma\in Z_{2n-4}(G(2n-4, 2n-1))$,
is given by the secant plane map
$\Sym^{2n-4}(P^1)\to G(2n-4, 2n-1)$, $D\mapsto W_D$, where $W_D\in\Sec_{2n-4}^{2n-5}(\G)$
 stands, as before, for the unique $(2n-5)$-dimensional $(2n-4)$-secant of $\Gamma$ with $W_D\cdot \G=D$.
The other ones, $B_\Lambda\in Z_{4n-8}(G(2n-4, 2n-1))$,
depend on a choice of an $(n-2)$-plane $\Lambda$ in $P^{2n-2}$; they are formed
by $(2n-5)$-planes meeting $\Lambda$ in codimension $\le 1$.
Note that $A_\Gamma$ and $B_\Lambda$ are of complementary dimensions,
$(2n-4)+(4n-8)=\dim G(2n-4, 2n-1)$, and, hence, their homology intersection number $[A_\Gamma]\cdot [B_\Lambda]$ is well defined
and does not depend on $\Lambda$.

If $\L\cap\G=\varnothing$, then  according to Proposition \ref{C-L-correspondence}
and Lemma \ref{MtoW}
a point $W\in A_\G\cap B_\L$ represents
an element $M=\pi_\L(W)$ of $\Sec^{n-3}_{2n-4}(C(\L))$ if $\L\not\subset W$ and
an element of $\Sec^{n-4}_{2n-4}(C(\L))$ if $\L\subset W$. In what follows, if $W$ is an isolated point of
$A_\G\cap B_\L$ (which is the case, in particular, for each $W\in A_\G\cap B_\L$, if
$C(\L)\notin \DPinf$), then we attribute to $M=\pi_\L(W)$
a positive integer {\it multiplicity} $m(M)$ equal, by definition, to the local intersection number of $[A_\Gamma]$ with $[B_H]$
at the point $W$.

\begin{lemma}\label{transversality-check}
Assume that $C\in\DP$ is a generic point,  $M\in \Sec_{2n-3}^{n-3}(C)$,
and $D=C\cdot M\in\Sym^{2n-3}(P^1)$, so that $M=M_D=\pi_\L(W_D)$ for $\L=\L(C)$.

Then $A_\G\cap B_\L$ is transverse at each of the $(2n-3)$ points $W_{D'}\in G(2n-4,2n-1)$, where
$D'\in \Sym^{2n-4}(P^1)$ is obtained from $D$ by dropping one of its points.
\end{lemma}

\begin{proof}
Let us choose an affine chart $\C^{2n-2}\subset P^{2n-2}$ generically with respect to $D=\{s_1,\dots, s_{2n-3}\}$ and $\Lambda$.
Consider a linear subspace $R\subset\C^{2n-2}$ intersecting $W_{D'}$, $D'=\{s_1,\dots, s_{2n-4}\}$,
transversally at one point, and denote by $\pi_R: \C^{2n-2}\to R$ the linear projection parallel to $W_{D'}$.
Then, we can naturally identify the tangent $(6n-12)$-dimensional space of $G(2n-4, 2n-1)$ at $W_{D'}$ with
the vector space $\Hom(W_{D'}^a,R)$ formed by affine maps from $W_{D'}^a=W_{D'}\cap \C^{2n-2}$ to $R$.
Its $(2n-4)$-dimensional subspace that is tangent to $A_\G$ at $W_{D'}$ is represented
by $f\in\Hom(W_{D'}^a,R)$ such that $f(s_i)\in L_i$ for $i=1,\dots,2n-4$, where $L_i=\pi_\R(T_{s_i}\G)$.
On the other hand, as soon as we pick $n-2$ points $q_{n-1}, \dots, q_{2n-4}$ in $\L\cap W_{D'}$ in a way that
$s_1, \dots, s_{n-2}, q_{n-1}, \dots, q_{2n-4}$  generate $W_{D'}$,
 the $(4n-8)$-dimensional tangent space to $B_\Lambda$ at $W_{D'}$ is formed by the affine maps $f\in\Hom(W_{D'}^a,R)$
such that $f(s_i)\in R$ for $i=1,\dots,n-2$
and $f(q_j)\in L$ for $j=n-1,\dots,2n-4$, where $L=\pi_R(\L)$
(it is a line since, by Lemma \ref{MtoW},  $\L\cap W_{D'}$ is of codimension $1$ in $\L$).

Under such a choice, the transversality of $A_\G$ and $B_\L$ at $W_{D'}$ means that for a
non-zero map $f\in \Hom(W_{D'}^a,R)$ it is impossible that $f(s_i)\in L_i$ for $i=1,\dots,2n-4$ and $f(q_j)\in L$
for $j=n-1,\dots,2n-4$.
Since transversality is an open condition and $\DP$ is irreducible (see Proposition \ref{D-reduced}), to show that
the transversality in question
holds for generic $\L$ (under the restriction that $\L\cap W_{D'}$ is of codimension $1$ in $\L$, which
is equivalent to $C\in\DP$) it is sufficient
to find just one example of $\L$ (satisfying the restriction on $\L\cap W_{D'}$, but not necessarily generic)
for which transversality holds.

Such $\L$ can be defined as the span of a generic line $L$ in $R$ and
the points $q_j\in W_{D'}$, $j=n-1,\dots,2n-4$, defined (preserving fixed the points $s_i$) by the conditions
$$
\begin{aligned}
s_{n-1}=&\frac1{n-1}(s_1+\dots+s_{n-2}+ q_{n-1})=
\frac1{n-1}(s_1+\dots+s_{n-3}+ s_{n}+q_{n})= \dots\\
=&\frac1{n-1}(s_1+s_n+ \dots+s_{2n-4}+ q_{2n-4})
\end{aligned}
$$
which guarantee, in particular, that
$s_1, \dots, s_{n-2}, q_{n-1}, \dots, q_{2n-4}$  generate $W_{D'}$.

Then, we get
$$
f(s_{n-1})=\frac1{n-1}(f(s_1)+\dots+f(s_{n-2})+  f(q_{n-1}))
$$
and,  for each $ i=2, \dots, n-2$,
$$
f(s_{n-1})=\frac1{n-1}(f(s_1)+\dots+f(s_{n-1-i})+f(s_n)+\dots + f(s_{n-2+i})+  f(q_{n-2+i})).
$$
Taking pairwise consecutive
differences and using the linear independence of non-zero vectors chosen on the lines
$L, L_k, L_l$ for each pair of $k,l$ (due to independence between $L_k, L_l$ and the genericity of $L$),
we deduce that
$f(s_{n-2})=f(s_n)=0$, $f(q_{n-1})=f(q_n)$, then that $f(s_{n-3})=f(s_{n+1})=0$, $f(q_{n})=f(q_{n+1})$ etc. up to $f(s_2)=f(s_n)=0$, $f(q_{2n-5})=f(q_{2n-4})$.
This implies $f(s_{n-1})=f(s_1)+f(q_{n-1})$, and we deduce from the linear independence of $L, L_{n-1}, L_1$ that $f(s_{n-1})=f(s_1)=f(q_{n-1})=0$.
Thus, $f=0$ and the transversality holds for our
choice of $\L$, as required.
\end{proof}

\begin{proposition}\label{Castelnuovo}
If $n\ge 3$, then:
\begin{enumerate}
\item
The subvarieties $\DP$ and $\DPinf$ are Zariski closed and $\P\sm(\DP\cup\DPinf)\ne\varnothing$.
\item For every $C\in\P\sm(\DP\cup\DPinf)$, the number of $M\in\Sec^{n-3}_{2n-4}(C)$ is finite
and
\begin{equation}\label{Castelnuovo-formula}
\sum_{ M\in\Sec^{n-3}_{2n-4}(C)} m(M)=\binom{n}2.
\end{equation}
\item
For a generic $C\in\P\sm(\DP\cup\DPinf)$, we have $m(M)=1$ for all $M\in\Sec^{n-3}_{2n-4}(C)$, and so
the latter set contains precisely
$\binom{n}2$ secants.
\item
For a generic point $C\in \DP$, the set $\Sec_{2n-3}^{n-3}(C)$ contains one secant and the set
$\Sec_{2n-4}^{n-3}(C)$ contains $\binom{n}2-(2n-3)$ secants, each secant from
$\Sec_{2n-4}^{n-3}(C)$
has multipllicity $1$ while the unique one, $M$, from $\Sec_{2n-3}^{n-3}(C)$ has $m(M)=2n-3$; furthermore, for the latter,
all the points in $M\cdot C$ are disctinct.
\end{enumerate}
\end{proposition}

\begin{proof}
In (1), closeness is evident, non-emptyness of $\P\sm\DP$ follows from Proposition \ref{D-reduced},
and $\P\sm\DPinf\ne\varnothing$ follows from the construction in the proof of (3).

The part (2) is a special case of the well-known Castelnuovo virtual count of secants (see, f.e., \cite{CurvesBook}),
except possibly positivity of multiplicities $m(M)$, which
is due to their
definition as local intersection numbers involved in $[A_\Gamma]\cdot[B_H]$.

To prove (3), pick $C\in \P\sm \DP$ and, in accordance with notation from Subsection \ref{via_projection},
identify $P^{n-1}$ with a projective subspace $H\subset (P^{2n-2} \sm \G)$
so that $C\in \P_H$.

By Kleiman's transversality theorem \cite{Kleiman},
there is a dense open set of linear transformations $g\in PGL_{2n-1}(\R)$
making $A_\Gamma$ transversal to $g(B_\Lambda)=B_{g(\Lambda)}$.
For generic $g$ near to the identity, $g(\Lambda)$ is still disjoint from $\Gamma\cup P^{n-1}$, and
$\pi_{g(\L)}(\G)$ is a curve $C_g\in \P_\Gamma$ that has $\Sec^{n-3}_{2n-4}(C_g)$ consisting of
$[A_\Gamma]\cdot[B_{g(\Lambda)}]$ elements, each of multiplicity $1$ due to transversality.
It remains to notice that having multiplicities $m(M)=1$ for all $M\in \Sec^{n-3}_{2n-4}(C_g)$
is an open condition on $C_g$
and that the Castelnuovo count gives $[A_\Gamma]\cdot[B_{g(\Lambda)}]=\binom{n}2$.

To prove (4), we proceed as before picking
a generic  $C\in \DP$,
so that by Lemma \ref{generic-on-DP}(2)-(3), the set $\Sec_{2n-3}^{n-3}(C)$ has only one element, namely, $M_D$
with the divisor $D=M_D\cdot C\in\Sym^{2n-3}(P^1)$ formed by distinct points.
We choose also coordinates in $P^{2n-2}$ so that $C\in \P_\Gamma$.

 Now we apply Kleimans's transversality theorem to
a Zariski open subset $Z_D\subset G(2n-4, 2n-1)$ formed by
$(2n-5)$-dimensional  projective spaces intersecting $W_D$
transversely and 
to
the group $G$ formed by linear projective transformations $g:P^{2n-2}\to P^{2n-2}$ such that $g(W_D)=W_D$.
By Lemma \ref{W-transversality}, $G$ acts on  $Z_D$ transitively.
Kleiman's theorem provides an open dense set in $G$ of $g\in G$ for which $g(B_\L)=B_{g(\L)}$ is
transversal to $A_\Gamma$ at all points of $Z_D$.
Thus, the transversality becomes achieved at all points of $A_\Gamma\cap B_{g(\Lambda)}$ except those $W\in A_\Gamma\cap B_{g(\Lambda)}$
that are subspaces $W\subset W_D$.

As a result, projecting $\Gamma$ from $g(\L)$ we get $C'$ with $\Sec_{2n-3}^{n-3}(C')$ consisting of one element of multiplicty $2n-3$
(see Lemma \ref{transversality-check}) and $\Sec_{2n-4}^{n-3}(C')$ consisting of
a certain number of elements of multiplicity 1. Due to Castelnuovo formula, this number is equal to $[A_\Gamma]\cdot[B_{g(\L)}]-(2n-3)=\binom{n}2-(2n-3)$.
Once more due to genericity of $\Lambda$ in $W_D$, for
the only $M\in \Sec_{2n-3}^{n-3}(C')$ all the points in $M\cdot C'$ are distinct.
\end{proof}

\subsection{The discriminant in the variety of multi-secants}\label{param-spaces}
In the variety of secants
$$
\MM=\{(C,M,D) | C\in\P\sm(\DPP\cup\DPPinf), M\in\Sec_{\ge2n-4}^{n-3}(C), D\in\Sym^{2n-4}(P^1), D\le M\cdot C \}
$$
we consider the {\it discriminant} (clearly, Zariski-closed)
$$\DM=\{(C,M,D) \in\MM\ |\ M\in\Sec_{\ge2n-3}^{n-3}(C)\}$$
and its complement
$$\M=\MM\sm\DM=\{(C,M,D)\in\MM\ |\ M\in\Sec_{2n-4}^{n-3}(C), D=M\cdot C \}.$$

\begin{proposition}\label{M-non-empty} If $n\ge3$, the variety $\M$ is non-empty.
Its projection
$$\pr_1:\M\to \P\sm(\DPP\cup\DPPinf),\ (C,M,D)\mapsto C,$$
is surjective, proper over $\P\sm(\DP\cup \DPinf)$
and sends $\DM$ onto $\DP\sm(\DPP\cup\DPPinf)$.
\end{proposition}

\begin{proof}
Non-emptiness follows from \ref{Castelnuovo}(1).
Surjectivity of $\pr_1$ over $\DP\sm\DPP $ follows
from Lemma \ref{syzygie}, over $\DPinf\sm\DPPinf$ follows directly from the definitions,
while surjectivity and properness over $\P\sm(\DP\cup \DPinf)$ from \ref{Castelnuovo}(2).
The last claim about $\DM\to\DPinf\sm(\DPP\cup \DPPinf)$ is straightforward from the definitions.

\end{proof}

\begin{lemma}\label{irreducibility} Each of the varieties $\MM$, $\M$ and $\DM$ is irreducible.
The first two of them are non-singular.
\end{lemma}

\begin{proof} As it follows from Lemma \ref{MtoW},
a relation $\dim(W_D\cap \L_C)=n-3$ holds for every $(C,M,D)\in \MM$.
Therefore, we have the following commutative diagram
$$\begin{CD}
\MM     @>\phi>>                \Cal N\\
@VV\pr_1 V                  @VVV\\
\P\sm(\DPP\cup\DPPinf) @>>> \P\sm(\DPP\cup\DPPinf)/\PGL_n
\end{CD}$$ where
$$\Cal N=\{(\L,D)\in G(n-1,2n-1)\times\Sym^{2n-4}(P^1)\ |\ \L\cap\G=\varnothing, \dim(W_D\cap \L)=n-3\}$$
and $\phi(C,M,D)=(\L_C,D)$.

First of all, note that $\phi:\MM\to \Cal N$ is a fibration with fibers $\PGL_n$,
and that it sends $\DM$ and its complement $\M$ to
$$\D^{\Cal N}=\{(\L,D)\in\Cal N\ |\ \exists D'\in\Sym^{2n-3}(P^1), D\le D', \L\subset W_{D'} \}$$
and its complement $\Cal N^*=\Cal N\sm\D^{\Cal N}$, respectively.

To check surjectivity of $\phi$, we pick $(\L,D)\in\Cal N$ and a projective subspace $H\subset P^{2n-2}$ of dimension
 $n-1$ which is disjoint from $\G$, and note
that $(\pi_\L(\G),\pi_\L(W_D),D)\in\MM$ lies in $\phi^{-1}(\L,D)$:
here, we use Lemma \ref{MtoW}(2)
that guarantees that $\pi_\L(\G)\notin\DPPinf$ since $\L\not\subset W_D$, and the condition $\pi_\L(\G)\notin\DPP$ is satisfied because
$W_D\cup \L$ generates a codimension $2$ subspace of $P^{2n-2}$.
Then, applying Proposition \ref{C-L-correspondence} we get the fibration property of $\phi$ stated above.

Next, note that
the image $\pr_2(\Cal N)$ of
the projection $\pr_2\!:\Cal N\to\Sym^{2n-4}(P^1)$ is Zarisky open in $\Sym^{2n-4}(P^1)$ and that $\pr_2$ provides
a fibration of $\Cal N$ over $\pr_2(\Cal N)$ with fiber
$G(n-2,2n-4)\times(P^n\sm F)$,
where $G(n-2,2n-4)$ stands for the choice
of $W_D\cap \L$ as a subspace of $W_D$ with given $D\in \Sym^{2n-4}(P^1)$,
$P^n$ stands for the choice of $\L$ with given hyperplane $W_D\cap \L$,
and $F$ is a Zariski-closed subset of $P^n$ determined by the conditions
$\L \subset W_D$ and  $\L\cap\Gamma\neq\emptyset$.
Since the base and the fiber are nonsingular and irreducible,  $\Cal N$
is non-singular and irreducible, as well as its Zariski-open subset $\Cal N^*$.
Because $\phi$ is a fibration with non-singular and irreducible fibers,
this implies that $\MM$ and $\M$ are non-singular and irreducible too.

Similarly, the irreducibility of $\DM$ follows from
that of $\D^{\Cal N}$.
Due to Lemma \ref{MtoW}
and the definition of ${\Cal N}$, the projective envelope of
$\L \cup W_D$ is of dimension $2n-4$, and by this reason the divisor $D'$ appearing in the definition of  $\D^{\Cal N}$
is unique. Thus, we have a well defined regular map
${\D^\Cal N}\to \Sym^{2n-3}(P^1)\times \G$, $(\L,D)\mapsto (D',D'-D)$.
It is surjective and has irreducible fibers that are open subsets of $G(n-1,2n-3)$
defined by the conditions
$\L\subset W_{D'}$, $\L\cap\G=\varnothing$, and $\L\not\subset W_D$.
\end{proof}

\begin{proposition}\label{codimensiontwo}
The subvariety $\DPinf\subset\P$ has codimension $\ge2$.
\end{proposition}

\begin{proof}
Proposition \ref{Castelnuovo}(1) implies that the codimension of $\DPinf$ in $\P$ is at least $1$,
while according to Proposition \ref{Castelnuovo}(2) the fibers of the projection $\pr_1\!:\MM\to \P\sm(\DPP\cup\DPPinf)$
over $\P\sm(\DP\cup\DPinf)\subset  \P\sm(\DPP\cup\DPPinf)$ are finite and non-empty, and according to Proposition \ref{Castelnuovo}(4)
the codimension of $\DPinf\cap \DP$ in $\P$ is $\ge 2$. Therefore, if $\DPinf$ had codimension $1$, then
$\DPinf\sm \DP$ would have codimension $1$, which would imply that $\MM$ is reducible, but it is not so
due to Lemma \ref{irreducibility}.
\end{proof}

\section{La strada di Segre}\label{def-S}

\subsection{Pencils of binary quadratic forms}\label{binary}
For any $(C,M,D)\in\MM$, we consider the pencil of hyperplanes $H_t\subset P^{n-1}$ such that $M\subset H_t$, $t\in P^1$.
The condition $C\notin\DPP$ guaranties that the intersections $H_t\cap C$ give a pencil of degree $2n-2$ divisors $D_t\in\Sym^{2n-2}(P^1)$,
$D_t\ge D$.
We can write $D_t=D+D^r_t$, where $\{D^r_t\}$ is the {\it residual pencil} of degree $2$ divisors.
In the case of $(C,M,D)\in \M$, the residual pencil $\{D^r_t\}$ is basepoint-free and so defines a double covering $P^1\to P^1$.
Its deck transformation and the two branch points, alias the fixed points of the deck transformation, will be called the {\it Segre involution} and {\it Segre points}, respectively.

The residual pencil of divisors $\{D^r_t\}$
can be seen
as a point
on the projective plane $P(V^*)$ where $V^*$ is dual to $V=H^0(P^1,\Cal O(2))$.
This yields
a map ${\Psi}:\MM\to P(V^*), (C,M,D)\mapsto\{D^r_t\}.$
By composing $\Psi$  with a polarity isomorphism
$\pol: P(V^*)\to P(V)$ we get a map $\widehat{\Psi}=\pol\circ\Psi :\MM\to P(V)$.

Recall that a {\it polarity isomorphism} identifies
a projective plane with its dual, via an auxiliary non-singular conic on the plane.
In our case, such a conic is the rational normal curve $\G\subset P(V^*)$ (cf.  Subsection \ref{via_projection})
whose points represent those pencils $\{D^r_t\}\in P(V^*)$ that have
a basepoint.
Note also that the projective plane $P(V)$
is canonically identified with $\Sym^2(P^1)$  so that
the conic polar to $\G$ becomes  the diagonal
$\D_2=\{\{x,x\}\ x\in P^1\}$.
In this terms,
$\widehat\Psi(C,M,D)\in P(V)=\Sym^2(P^1)$
is the pair of the Segre points
of $\{D^r_t\}$, if $\{D^r_t\}\in P(V^*)\sm\G$, or, equivalently, if $\widehat\Psi(C,M,D)\in P(V)\sm \D_2$.

\begin{lemma}\label{discr-criterion}
For any $(C,M,D)\in\MM$, the condition $(C,M,D)\in\DM$ is equivalent to $\widehat\Psi(C,M,D)\in \D_2$.
Furthermore, if $\widehat\Psi(C,M,D)=\{x,x\}\in\Sym^2(P^1)=\D_2$, then $x$ is a basepoint of the pencil $\{D^r_t\}$.
\end{lemma}

\begin{proof}
By definition, $(C,M,D)\in\DM$ if and only if $D< M\cdot C$
which, in its turn, holds if and only if the residual pencil $D^r_t=\Psi(C,M,D)$ contains a basepoint
(namely, $M\cdot C-D$, which is a divisor of degree $1$, since $C\notin\DPPinf$).
The latter means that $\Psi(C,M,D)\in\G$,
and thus, $\widehat\Psi(C,M,D)\in\D_2$. The second statement holds since for each $\{x,x\}\in\Sym^2(P^1)=\D_2$
the pencil $\pol^{-1}\{x,x\}$ seen as a line in $P(V)$ is formed by $p\in H^0(P^1,\Cal O(2))$ vanishing at $x$.
\end{proof}

\begin{proposition}\label{wall-transversality}
The map $\widehat\Psi:\MM\to P(V)$ is transverse to
$\D_2$ at generic points of $\DM$.
\end{proposition}

\begin{proof}
Note that transversality even at one point of $\DM$ guarantees transversality at a generic point of
$\DM$, since $\DM= \widehat\Psi^{-1}(\D_2)$ (see Lemma \ref{discr-criterion}), $\DM$ is irreducible (see Lemma \ref{irreducibility}),
and transversality is an open condition.

We start with a generic point of $\DP\sm \DPinf\subset\P$ represented by a curve $C_0=[p_{1,0}:\dots:p_{n,0}]$
and define its variation $C_t=[p_{1,t}:\dots:p_{n,t}]$ in a particular way.
Namely, pick a multisecant $M\in\Sec^{n-3}_{2n-3}(C_0)$ (see Lemma \ref{generic-on-DP}(2) for its existence and uniqueness) and choose coordinates $[x_1:\dots:x_n]$ in $P^{n-1}$ so that
$M=\{x_1=x_2=0\}$. Then, the polynomials $p_{1,0}$ and $p_{2,0}$ have  $2n-3$ common roots forming the divisor
$M\cdot C_0$. Take a divisor $D < M\cdot C_0$ of degree $2n-4$ obtained by dropping one of these roots and factorize the polynomials as
$p_{1,0}=(u-av)(u-bv)r$ and $p_{2,0}=(u-av)(u-cv)r$
where the common factor $r$ has $D$ as the zero divisor and
$a,b,c\in\C$ are pairwise distinct.

We define a variation $C_t$, $t\in\C$,  by letting $p_{1,t}=
(u-(a+t)v)(u-bv)r$, and leaving  unchanged $p_{i,t}=p_{i,0}$ for $i=2,\dots,n$.
Then $M\in\Sec^{n-3}_{2n-4}(C_t)$ for all $t$, and  $M\cdot C_t=D$ for $t\ne0$.
According to this and a generic choice of $C_0\in \DP\sm \DPinf$,
we get $(C_0,M,D)\in \DM$ and $(C_t,M,D)\in \M$ for $t\ne0$.
The points $\widehat \Psi(C_t,M,D)$ form a line  (linearly parameterised by $t$) in $P(V)$ which is transversal to $\D_2$, since it intersects $\D_2$ at two points, $t=0$ and $t=c-a$.
\end{proof}

\subsection{Segre weights and indices}\label{Segre-weight}
In the real setting, the real locus  $(\D_2)_\R$ of the conic $\D_2$ divides the real locus $P(V)_\R$ of the plane $P(V)$
into two connected components, and we introduce the {\it index  function}
$${\rm ind}\!:P(V)_\R\sm (\D_{2})_\R\to\{+1,-1\}
$$
that takes value $1$ in the exterior (M\"obius band component) and $-1$ in the interior
(disc component) of $(\D_2)_\R$.

By Lemma \ref{discr-criterion}, $\Psi(C,M,D)\in P(V)_\R\sm (\D_{2})_\R$ if $(C,M,D)\in\M_\R$, and we define
the {\it Segre weight} of $(C,M,D)$ to be
$$\Sw(C,M,D)={\rm ind}(\Psi(C,M,D)).$$
If $\Sw(C,M,D)=1$ (that is, if $\Psi(C,M,D)$ belongs to the exterior of $(\D_2)_\R$), we call $M$, and its residual pencil, {\it hyperbolic}.
If $\Sw(C,M,D)=-1$ (that is, if $\Psi(C,M,D)$ belongs to the interior), we call them {\it elliptic}. By virtue of identification of
$P(V)$ with $\Sym^2(P^1)$, this definition can be rephrased in terms of Segre points: the residual pencil is hyperbolic, if the both Segre points
 are real, and elliptic, if they are imaginary conjugate.

\begin{lemma}\label{continuity}
$\Sw$ is continuous (locally constant) on $\M_\R$.
\end{lemma}

\begin{proof}
It follows
from continuity of the roots of a polynomial
as functions of the coefficients.
\end{proof}

\begin{cor}\label{pushing-forward} The push-forward of $\Sw$ from $\M_\R$ to $\PR\sm(\DPR\cup\DPinf_\R)$ defined,
for $C\in \PR\sm(\DPR\cup\DPinf_\R)$, by
\begin{equation}\label{function-S}
S^\P(C)=\prod_{M\in\Sec^{n-3}_{2n-4}(C)}\Sw(C,M, M\cdot C)^{m(M)}\in\{+1,-1\},
\end{equation}
{\rm (where $m(M)$ is the multiplicity of $M$ as it appears in the Castelnuovo formula
(\ref{Castelnuovo-formula}))}
is continuous
on $\PR\sm(\DPR\cup\DPinf_\R)$.
\end{cor}

\begin{proof}
According to Lemma \ref{pushing-forward}, $\Sw$ is constant along the connected components  of $\M_\R$, denoted below
$X_i$, $i=1,2,\dots$.
By Lemma \ref{irreducibility}
and Proposition \ref{M-non-empty}, $\MM$ is non-singular and the projection map $\pr_1\!:\MM\to \P\sm(\DPP\cup\DPPinf)$ is proper over
$\P\sm(\DP\cup\DPinf)$. Thus,
$S(C)$ is equal to the product of $\Sw(X_i)^{d(\pr_1\vert_{X_i},C)}$
where the product is taken over all
$i$ and  $d(\pr_1\vert_{X_i},C)$ denotes the local $(\rm{mod\ 2})$-degree at $C$ of $\pr_1$ restricted to $X_i$. It remains to notice that,
due to properness of $\pr_1$ over $\P\sm(\DP\cup\DPinf)$,
each of $d(\pr_1\vert_{X_\alpha},C)$ is locally constant along  $\P\sm(\DP\cup\DPinf)$.
\end{proof}

\begin{proposition}\label{S-continuity}
The function $S^\P$ defined by (\ref{function-S}) on $\PR\sm(\DPR\cup\DPinf_\R)$
can be extended by continuity to a function $S^\P:\PR\sm\DPR\to\{\pm1\}$.
\end{proposition}

\begin{proof}
Since $\P$ is smooth and, due to Corollary \ref{codimensiontwo}, $\DPinf$ has codimension $\ge2$,
the function $S$, as any locally constant function on $\PR\sm(\DPR\cup\DPinf_\R)$, extends by continuity to $\PR\sm\DPR$.
\end{proof}

\subsection{Wall-crossing}\label{Wall-crossing}
\begin{lemma}\label{alternation-in-M}
Assume that a path $(C_t,M_t,D_t)\in\MM_\R$, $t\in[a,b]$, intersects $\DM_\R$ transversely at a generic point of  $\DM_\R$.
Then $\Sw(C_t,M_t,D_t)\in\{\pm1\}$ alternates at the point of crossing  with $\DM_\R$.
\end{lemma}

\begin{proof}
By Proposition \ref{wall-transversality}, the path $\widehat\Psi(C_t,M_t,D_t)\in P(V)$ intersects
$(\D_{2})_\R$
transversely, and
therefore
$\Sw(C_t,M_t,D_t)$ alternates after the index function.
\end{proof}

\begin{proposition}\label{wall-crossing}
If a path $C_t\in\PR\sm\DPinf_\R$, $t\in[a,b]$, crosses transversely a wall of $\DPR$
at a generic point $C_0$, $0\in(a,b)$, then $S^\P(C_t)$ alternates at $t=0$.
\end{proposition}

\begin{proof}
According to Proposition \ref{Castelnuovo}(4), the preimage of $C_0$ in $\DM$
consists of $2n-3$ points $(C_0,M,D)$
represented by pairwise distinct
$D\in \Sym^{2n-4}(P^1)$ with $D<M\cdot C_0$ where
$M$ is the unique element of $\Sec^{n-3}_{2n-3}(C_0)$.
Since $C_0\in\DP_\R$, 
among these $2n-3$ points an odd number belong to $\DM_\R$. Respectively, since, in addition, $\MM$ and $\DM$ are non-singular at $C_0$,
the path $C_t$
is covered in $\MM_\R$ by an odd number of paths that all cross $\DM_\R$,
so that
we can apply Lemma \ref{alternation-in-M}
and conclude that the sign of $S^\P(C_t)$ alternates at $t=0$.
\end{proof}

For $X\in\XR\sm\DXR$ we let $S(l,X)=S^\P(\Jet^1(l,X))$
and call $S(l,X)$ the
 {\it Segre index} of $l$ in $X$.

\begin{proposition}\label{wall-crossing-in-X}
$S\!:\XR\sm\DXR\to\Z/2$ is a continuous function that alternates its value  under generic crossings of the walls of $\DXR$.
\end{proposition}
\begin{proof} Straightforward from
Propositions \ref{wall-crossing} and \ref{discr-comparison}.
\end{proof}

\begin{lemma}\label{coincidence-example}
For some $X\in\XR\sm\DXR$ we have
$S(l,X)=\ind(l,X)$.
\end{lemma}

\begin{proof} Consider the same pair $(l,X)$ with $\ind(l,X)=1$ as in the proof of Proposition \ref{one-example}.
Recall that in that example
 $C=\Jet^1(l,X)\!:P^1\to P^{n-1}$  is composed of a degree 2
map $P^1\to P^1$, $[u:v]\mapsto[u^2:v^2]$, followed by an embedding
$B\!:P^1\to P^{n-1}$, $[u:v]\mapsto [u^{n-1}:u^{n-2}v:\dots:v^{n-1}]$.

For any $D\in \Sym^{n-2} (P^1)$,
there exists one and only one $(n-3)$-dimensional $(n-2)$-secant $M_D$ of $B$ with $M_D\cdot B=D$.
The same $M_D$ is a $(2n-4)$-secant of $C$ with $M_D\cdot C=2D$.
Furthermore, since $B$ is not contained in any hyperplane of $P^{n-1}$,
it can not have any $(n-1)$-secant, and, as a consequence, $C\in\DPinf_\R\sm\DP_\R$.
Hence, the pencil
$\Psi(C,M_D,2D)\in P(V)$ is well-defined for any $D\in \Sym^{n-2} (P^1)$.
This pencil is hyperbolic because the ramification points of the map $[u:v]\mapsto[u^2:v^2]$ are real.
By continuity of $S^\P$,
this implies that any small perturbation
$C'\in \P_\R\sm(\DP_\R\cup \DPinf_\R)$
of $C$ may have only hyperbolic real $(n-3)$-dimensional $(2n-4)$-secants
and therefore has  $S^\P(C')$ equal to $+1$.

Finally, there remain to pick a  perturbation $(l, X')$ of $(l,X)$ with $\Jet^1(l,X') =C'$
and to use the continuity of $\ind(l,X)$ (see \ref{continuity-W}).
\end{proof}

\begin{theorem}\label{Euler=Segre}
For any $X\in\XR\sm\DXR$ we have
$S(l,X)=\ind(l,X)$.
\end{theorem}

\begin{proof}
For $n=2$, see \cite{abundance}.
For $n\ge 3$,  due to
Propositions \ref{wall-crossing-of-I-and-W},
\ref{wall-crossing-in-X} and Lemmas \ref{connected}, \ref{continuity-W}, it is  sufficient
to check coincidence of $\ind$ and $S^\X$ for one particular example,
what is done in Lemma \ref{coincidence-example}.
\end{proof}

\section{Another viewpoint on the generalized Segre indices}\label{another-viewpoint}

\subsection{The case of quintic threefolds}\label{Cremona}
If $n=3$, then $X$ is a quintic threefold in $P^4$ and the curve $C=\Jet^1(l,X) : P^1\to P^2$ that we associate
with a pair $(l,X)$ is a parametrized rational plane quartic.
For a generic $X\in \X$, the singular locus of $C$ consists of three nodal
points $r_1, r_2, r_3$.
Applying an elementary quadratic (Cremona) transformation $\crem:P^2\dashrightarrow P^2$ based at
these three points, we obtain a conic
$Q=\crem(C)$ and
a triple of points $B=\{s_1,s_2,s_3\}\subset P^2\sm Q$ that are the base points of the inverse quadratic transformation.
 In the real setting, the conic $Q$ must have $Q_\R\ne\varnothing$ (since it inherits from $C$ a real parametrization by
 $l$ and $l_\R\ne\varnothing$),
while among the
points $s_i$ either all three are real
(if the nodes of $C$ are real) or one is real and two others are imaginary complex conjugate.

\begin{proposition}\label{Segre-for quartics}
For a generic real rational plane quartic $C$,
the Segre index
$S^\P(C)=S(l,X)$ is equal to $(-1)^{int(B,Q)}$, where $int(B,Q)$ is the number of real points in
$B$ that lie inside $Q_\R$.
\end{proposition}

\begin{proof}
For $n=3$ we have $\Sec^{n-3}_{2n-4}(C)=\{r_1,r_2,r_3\}$ and, for each $i=1,2,3$, the pencil cut on $C$
by the lines through $r_i$ is transformed by $\crem$ into the pencil cut on $Q$ by the lines through $s_j=\crem(\la r_k, r_l\ra)$ with $k,l\ne i$. The latter pencil is elliptic if and only if $s_j$ is real and lie inside $Q_\R$.
\end{proof}

The Segre index $S^\P(C)$ of a parametrized real quartic $C\!:P^1\to P^2$ can be also calculated using its {\it chord diagram}.
Such a diagram, $D(\g)$,
defined for any
curve $\g:S^1\to F$ regularly immersed into a surface $F$,
is usually presented by a circle $S^1$ with
the preimages of each double point connected by a chord in the 2-disc $D^2$ bounded by $S^1$.
Some of these chords intersect
and, by definition,  the index $\rm{ind}(D(\g))$ of
$D(\g)$ is the number of such intersecting pairs.

In our case, $S^1=P^1_\R$ and $\gamma = C\vert_{P^1_\R} : P^1_\R\to P^2_\R$.
For a generic $C$, the latter is a regular immersion and a chord of $D(\g)$
connects $t_1,t_2\in P^1_\R$ if $C(t_1)=C(t_2)$ is a self-intersection point of the real locus of $C$,
in other words, a real {\it cross-like node}. In addition, $C$ may have nodes 
with 
$t_1,t_2\in P^1\sm P^1_\R$.
Namely, either a real node called a {\it solitary point} where $C(t_1)=C(t_2)$ with $t_2=\bar t_1$, or
a pair of conjugate imaginary nodes: $C(t_1)=C(t_2)\in P^2_\C \sm P^2_\R$ for one node
and $C(\bar t_1)=C(\bar t_2)$ for the conjugate one. If, in the latter case, the points $t_1, t_2$ (and, thus, also
the points $\bar t_1$, $\bar t_2$) lie in the same connected component of $P^1_\C\sm P^1_\R$, we call such an imaginary pair of nodes {\it
essential}, and otherwise {\it inessential}. We denote by $\rm{ind}_{im}(C)$ the number of essential pairs.

\begin{proposition}\label{Segre-for quartics-bis}
For a generic real rational plane quartic $C\in\P_\R$
the Segre index $S^\P(C)=S(l,X)$ is equal to $(-1)^{\rm{ind}(C)+\rm{ind}_{im}(C)}$.
\end{proposition}

\begin{proof}
Let us use the real locus $Q_\R$ of the conic $Q=\crem(C)$
as the circle of the diagram $D(C\vert_{P^1_\R})$. Then, each chord of the diagram becomes an interval of a real line
connecting a pair of points from $B=\{s_1,s_2,s_3\}$.

Choose some real point, say $s_1$, then points $s_2$, $s_3$ are either also real, or conjugate imaginary.
In the first case, lines $s_1s_2$ and $s_1s_3$ are real and form an intersecting pair of chords if and only if
$s_1$ lies inside $Q_\R$. In the second case, these two lines are imaginary and each of them intersects $Q$ at a pair of points
that both lie in
the same connected component of $Q\sm Q_\R$ if and only if $s_1$ lies inside $Q_\R$.
So,  $int(B,Q)=\rm{ind}(C)+\rm{ind}_{im}(C)$ and it remains
to apply Proposition \ref{Segre-for quartics}.
\end{proof}

\subsection{The case of septic 4-folds}
If $n=4$,
then $X$ is a  hypersurface of degree 7 in $P^5$ and the curve $C=\Jet^1(l,X)$ that we associate
with a pair $(l,X)$ is a parametrized rational sextic in $P^3$.
For a generic pair $(l,X)\in\X$, the following properties hold:
(1) $C\in\P\sm(\DP\cup\DPinf)$, (2) $C$ is non-singular, and (3) among the quadrisecants to $C$ there are no multiple ones,
so that $C$ has precisely 6 distinct quadrisecants, $M_i$, $i=1,\dots,6$ (cf.  Proposition \ref{Castelnuovo}).

Since through any 19 points in $P^3$ one can trace a cubic surface, we may pick 19 points on $C$ and find a cubic surface $F$ containing them.
Then, according to Bezout theorem, $F$ should contain $C$. Furthermore, again by Bezout theorem,
$F$ should contain also all the 6 quadrisecants.
Finally, it is also not difficult to show that assumptions (1)--(3) imply that
(4) $F$ is non-singular, and $M_i$, $i=1,\dots,6$, are pairwise disjoint . After that,
by Schl\"afli theorem, there exist 6 other lines $\hat M_i\subset F$ which extend the 6-tuple of lines $M_i$ to a double six. Contraction of the lines $\hat M_i$, $i=1,\dots,6$ gives a plane $P^2$ and
a standard lattice calculation (taking into account (2) -- (4)) shows that
the sextic $C$ is transformed 
into
a conic $Q\subset P^2$
disjoint from the
base-point set $B$ of the contraction and the lines $M_i$, $i=1,\dots,6$, are transformed into 6 conics each passing through all but one points of $B$.

\begin{proposition}\label{Segre-for-sextics}
For a real rational sextic $C\subset P^3$ satisfying the above genericity conditions (1)--(3),
the Segre index $S^\P(C)$ is equal to $(-1)^{int(B,Q)}$, where $int(B,Q)$ is the number of real points from $B$ that lie inside $Q_\R$.
\end{proposition}

The proof is analogous to that of Proposition \ref{Segre-for quartics}.

\subsection{Generalization} The above results extend from $n=3$ and $4$ to any $n$ in the following way.

First of all, one can show that for any $n$ a generic real rational curve $C$ of degree $2n-2$ in $P^{n-1}$
lies on a rational surface $F(C)\subset P^{n-1}$ of degree $\binom{n-1}2$. To construct such pairs $(C, F(C))$ we
choose  a set $B\subset P^2$ of $\binom{n}2$ points in general position and consider the linear system
of degree $n-1$ curves in $P^2$ passing through $B$. This linear system is of
dimension $\frac{(n-1)(n+2)}2-\binom{n}2=n-1$. If $n\ge 4$, it defines an embedding $g_B:P^2(B)\to P^{n-1}$  where $P^2(B)$
is the plane blown up in $B$ (for $n=3$, instead of an embedding it gives an elementary Cremona transformation described in Subsection \ref{Cremona}).  The
degree of the image  is $(n-1)^2-\binom{n}2=\binom{n-1}2$.
Finally, we pick a conic $Q\subset P^2$ disjoint from $B$
and put $C=g_B(Q)$, $F(C)=g_B(P^2(B))$.
The fact that thus obtained $C$ is generic follows from an appropriate dimension count.

For each $s\in B$, there is a unique degree $n-2$ curve $A_s$ that passes through all the points of $B\sm\{s\}$.
Each of the $(n-3)$-dimensional $(2n-4)$-secants of $C$ can be obtained by taking a linear projective span of $g_B(A_s)\subset P^{n-1}$
for some $s\in B$ (these spans are dual to pencils of curves of degree $n-1$ that have $A_s$ as the fixed part and lines through $s$ as the moving part).
Finally, the same arguments as in the proof of Proposition  \ref{Segre-for quartics} give the following result.

\begin{proposition}
For a generic real rational degree $2n-2$ curve $C$ in $P^{n-1}$,
the Segre index $S^\P(C)$ is equal to $(-1)^{int(B,Q)}$, where $int(B,Q)$ is the number of real points from $B$ that lie inside $Q_\R$.
\qed
\end{proposition}

\end{document}